\documentclass[12pt,a4paper]{article}
\usepackage{amsfonts,amsmath,amsthm,amssymb,dsfont,mathrsfs} 
\usepackage{array}
\usepackage{bm}
\usepackage{txfonts}
\usepackage{graphicx}
\usepackage{indentfirst} 
\usepackage{latexsym}
\usepackage{longtable,multirow}
\usepackage{CJK} 
\usepackage{color}
\theoremstyle{plain}
\newtheorem{theorem}{Theorem}

\newtheorem{lemma}[theorem]{Lemma}

\newtheorem{conjecture}[theorem]{Conjecture}

\newtheorem{\proofname}{Proof}

\long\def\delete#1{}
\long\def\ps{1.6}
\def\ch {{\rm ch}}
\def\C {{\mathcal C}}
\title{\bf Edge Roman domination on graphs}
\author{Gerard J. Chang$^{1 2}$\thanks{E-mail: gjchang@math.ntu.edu.tw. Supported in part by the Ministry of Science and Technology under grant NSC101-2115-M-002-005-MY3.} ~~~
        Sheng-Hua Chen$^1$\thanks{E-mail: b91201040@ntu.edu.tw.} ~~~
        Chun-Hung Liu$^{3}$\thanks{E-mail: cliu87@math.gatech.edu. This work is based on discussions when the author visited the National Center for Theoretical Sciences, Taipei Office.} \\
 {\footnotesize $^1$Department of Mathematics, National Taiwan University, Taipei 10617, Taiwan} \\[-0.2cm]
 {\footnotesize $^2$National Center for Theoretical Sciences, Taipei Office, Taipei 10617, Taiwan} \\[-0.2cm]
 {\footnotesize $^3$School of Mathematics, Georgia Institute of Technology, Atlanta, Georgia 30332, USA} \\[-0.2cm]
}
\date{August 10, 2014} 

\begin{document}
\maketitle

\begin{abstract}
An edge Roman dominating function of a graph $G$ is a function $f\colon E(G) \rightarrow \{0,1,2\}$ satisfying the condition that every edge $e$ with $f(e)=0$ is adjacent to some edge $e'$ with $f(e')=2$.
The edge Roman domination number of $G$, denoted by $\gamma'_R(G)$, is the minimum weight $w(f) = \sum_{e\in E(G)} f(e)$ of an edge Roman dominating function $f$ of $G$.
This paper disproves a conjecture of Akbari, Ehsani, Ghajar, Jalaly Khalilabadi and Sadeghian Sadeghabad stating that if $G$ is a graph of maximum degree $\Delta$ on $n$ vertices, then $\gamma_R'(G) \le \lceil \frac{\Delta}{\Delta+1} n \rceil$.
While the counterexamples having the edge Roman domination numbers $\frac{2\Delta-2}{2\Delta-1} n$, we prove that $\frac{2\Delta-2}{2\Delta-1} n + \frac{2}{2\Delta-1}$ is an upper bound for connected graphs.
Furthermore, we provide an upper bound for the edge Roman domination number of $k$-degenerate graphs, which generalizes results of Akbari, Ehsani, Ghajar, Jalaly Khalilabadi and Sadeghian Sadeghabad.
We also prove a sharp upper bound for subcubic graphs.

In addition, we prove that the edge Roman domination numbers of planar graphs on $n$ vertices is at most $\frac{6}{7}n$, which confirms a conjecture of Akbari and Qajar.
We also show an upper bound for graphs of girth at least five that is 2-cell embeddable in surfaces of small genus.
Finally, we prove an upper bound for graphs that do not contain $K_{2,3}$ as a subdivision, which generalizes a result of Akbari and Qajar on outerplanar graphs.

\end{abstract}

\section{Introduction}

The articles by ReVelle \cite{r1, r2} in the Johns Hopkins Magazines suggested a new variation of domination called Roman domination, see also \cite{rr} for an integer programming formulation of the problem.
Since then, there have been several articles on Roman domination and its variations, such as \cite{ahlzs, ckpw, cl, hv}.
Emperor Constantine had the requirement that an army or legion could be sent from its home to defend a neighboring location only if there was a second army which would stay and protect the home.
Thus, there are two types of armies, stationary and traveling.
Each vertex (city) has no army must have a neighboring vertex with a traveling army.
Stationary armies then dominate their own vertices, and a vertex with two armies is dominated by its stationary army, and its open neighborhood is dominated by the traveling army.

We may formulate the problem in terms of graphs.
Graphs are simple in this paper.
A {\it Roman dominating function} of a graph $G$ is a function $f\colon V(G) \to \{0, 1, 2\}$ such that every vertex $v$ with $f(v)=0$ is adjacent to some vertex $u$ with $f(u) = 2$.
The {\it weight} of a Roman dominating function $f$ is the value $w(f) = \sum_{v\in V(G)} f(v)$.
The {\it Roman domination number} of $G$, denoted by $\gamma_R(G)$, is the minimum weight of a Roman dominating function of $G$.

Recently, Roushini Leely Pushpam and Malini Mai \cite{Pushpam} initiated the study of the edge version of Roman domination.
An {\it edge Roman dominating function} of a graph $G$ is a function $f\colon E(G) \rightarrow \{0,1,2\}$ such that every edge $e$ with $f(e)=0$ is adjacent to some edge $e'$ with $f(e')=2$.
The {\it weight} of an edge Roman dominating function $f$ is the value $w(f) = \sum_{e\in E(G)} f(e)$.
The {\it edge Roman domination number} of $G$, denoted by $\gamma'_R(G)$, is the minimum weight of an edge Roman dominating function of $G$.
In fact, the edge Roman domination number of $G$ equals the Roman domination number of its line graph.
However, we are interesting in finding upper bound of $\gamma'_R(G)$ in terms of $\lvert V(G) \rvert$ instead of $\lvert E(G) \rvert$.
So reducing problem to the line graph is usually not helpful to obtain a non-trivial upper bound.

Roushini Leely Pushpam et al.~\cite{Pushpam} established some properties of edge Roman dominating functions and determined the edge Roman dominating numbers of paths and cycles: $\gamma'_R(P_n)=\lfloor \frac{2n}{3}\rfloor$ and $\gamma'_R(C_n)=\lceil \frac{2n}{3}\rceil$.
Akbari et al.~\cite{Akbari1} gave an upper bound for a graph in terms of its maximum degree and order: $\gamma'_R(G) \leq \frac{2\Delta}{2\Delta+1}n$ for graphs $G$ of maximum degree $\Delta$ on $n$ vertices.
They then conjectured the following.

\begin{conjecture} {\bf \cite{Akbari1}} \label{conj1}
  If $G$ is a graph of maximum degree $\Delta$ on $n$ vertices, then $\gamma'_R(G)\leq \lceil\frac{\Delta}{\Delta+1}n \rceil.$
\end{conjecture}

They also established several results for special graphs as follows.
For a graph $G$ of maximum degree $\Delta$ on $n$ vertices, if $G$ has a perfect matching, then $\gamma'_R(G)\leq \frac{2\Delta-1}{2\Delta}n$.
If $T$ is a tree of $n$ vertices, then $\lceil \frac{2(n-\ell(T)+1)}{3}\rceil \leq \gamma'_{R}(T)\leq \lceil \frac{2(n-1)}{3}\rceil = \lfloor \frac{2n}{3} \rfloor$ where $\ell(T)$ is the number of leaves, and the equality holds if and only if $T=P_n$.
If $n\geq 2$, then $\gamma'_R(P_2 \Box P_n)=\lceil \frac{4n}{3}\rceil$ and $\gamma'_R(P_3 \Box P_n)=2n$.
If $n\geq 1$, then $\gamma'_R(Q_n)\geq \frac{2^{n+1}n}{3n-1}$.
Akbari et al.~\cite{Akbari2} gave the following two results on planar graphs.
If $G$ is outerplanar, then $\gamma'_{R}(G)\leq \frac{4}{5}n$.
If $G$ is planar and claw-free, then $\gamma'_{R}(G)\leq \frac{6}{7}n$.
They conjectured that the claw-freeness in the above result can be removed.

\begin{conjecture} {\bf \cite{Akbari2}} \label{conj2}
  If $G$ is a planar graph of $n$ vertices, then $\gamma'_R(G)\leq \frac{6}{7}n.$
\end{conjecture}

We address extremal problems on edge Roman domination in this paper.
We disprove Conjecture \ref{conj1} in Section 2 and prove an essentially tight upper bound for $k$-degenerate graphs in Section 3.
More precisely, we prove that $\gamma'_R(G) \leq \frac{2k}{2k+1}\lvert V(G) \rvert$ for $k$-degenerate graphs $G$, and $\gamma'_R(G) \leq \frac{2\Delta-2}{2\Delta-1} \lvert V(G) \rvert + \frac{2}{2\Delta-1}$ for graphs $G$ of maximum degree $\Delta$.

In Section 4, we prove that $\gamma_R'(G) \leq \frac{4}{5} \lvert V(G) \rvert$ for subcubic graphs $G$ other than $K_{3,3}$.
This bound is attained by infinitely many graphs.
Furthermore, this result not only improves the mentioned result when $\Delta=3$ but also is a preparation for a result in the next section.

In Section 5, we confirm Conjecture \ref{conj2} and show that the same upper bound holds for graphs 2-cell embeddable in the plane or the projective plane.
We then improve the upper bound for graphs of girth at least five that can be drawn in surfaces of small genus.
The second result takes the advantage of the result on subcubic graphs in Section 4.

Finally, in Section 6, we prove that $\gamma'_R(G) \leq \frac{4}{5}\lvert V(G) \rvert$ for graphs that do not contain a subgraph isomorphic to a subdivision of $K_{2,3}$, which generalizes a result of Akbari et al.~\cite{Akbari2} on outerplanar graphs. 
Note that $C_5$ attains the bound $\frac{4}{5}n$, and the coefficient $\frac{4}{5}$ of $n$ cannot be improved by excluding finitely many graphs: let $G$ be the graph obtained from the disjoint union of $k$ $5$-cycles by adding a vertex adjacent to a vertex of each $5$-cycle, then $\gamma'_R(G) =4k = \frac{4}{5} |V(G)| - \frac{4}{5}$.
We will prove that the mentioned example is more or less the only example for graphs that attain this coefficient of $n$.
More precisely, we shall prove that the upper bound can be improved if no $5$-cycle in the graph can be separated from the rest of the graph by deleting at most one vertices.

Now we fix some notation that will be used in the rest of this paper.
Let $G$ be a graph.
For every $X \subseteq V(G)$, we define $N(X)$ to be the set of vertices of $G-X$ adjacent to a vertex in $X$, and we define $N[X]$ to be $N(x) \cup X$.
When $X$ consists of only one vertex $v$, we denote $N(X)$ and $N[X]$ by $N(v)$ and $N[v]$, respectively.
In a graph $G$, for a subset $S \subseteq V(G)$ the {\it subgraph induced by $S$} is the graph $G[S]$ with vertex set $S$ and edge set $\{xy \in E(G)\colon x,y \in S\}$.
The {\it deletion} of $S$ from $G$, denoted by $G-S$, is the induced subgraph $G[V(G)-S]$.
A {\it matching} $M$ of $G$ is a subset of edges of $G$ such that no two edges in $M$ are adjacent.
The set of all end vertices of the edges in $M$ is denoted by $V(M)$.
A subset of vertices is {\it stable} if every pair of vertices in the set are non-adjacent.
For every integer $k$, we say that $G$ is {\it $k$-degenerate} if every subgraph of $G$ contains a vertex of degree at most $k$.

\section{Counterexamples to Conjecture \ref{conj1}}

This section constructs counterexamples to Conjecture \ref{conj1}.
We first consider the complete bipartite graph $K_{r,s}$ with partite sets $X = \{x_1, x_2, \ldots, x_r\}$ and $Y = \{y_1, y_2, \ldots, y_s\}$.


\begin{theorem} \label{bipartite}
  If $1 \le r \le s$, then $\gamma_R'(K_{r,s})=2r$ for $r<s$ and $\gamma_R'(K_{r,s})=2r-1$ for $r=s$.
\end{theorem}

\begin{proof}
For $r<s$, the function $f$ defined by $f(x_iy_i)=2$ for $1 \le i \le r$ and $f(x_iy_j)=0$ for all other edges $x_iy_j$ is an edge Roman dominating function of weight $2r$, which gives $\gamma'_R(K_{r,s})\le 2r$.
For $r=s$, a modification on $f(x_r y_r)=1$ gives that $\gamma'_R(K_{r,s})\le 2r-1$.

On the other hand, suppose $f$ is an edge Roman dominating function of $K_{r,s}$ with the minimum weight.
Assume there are $a$ edges $e$ with $f(e)=2$.
If $a \ge r$, then $\gamma'_R(K_{r,s}) \geq w(f) \ge 2a \ge 2r$, and we are done.
So we may assume that $a<r$.
Then $X$ contains at least $r-a$ vertices and $Y$ contains at least $s-a$ vertices that are not incident to any edge $e$ with $f(e)=2$.
Hence there are $(r-a)(s-a)$ edges $e'$ having $f(e')=1$.
These give $w(f) \ge 2a+(r-a)(s-a)$.

If $r<s$, then $s-a \ge 2$ and so $\gamma_R(K_{r,s}) = w(f) \ge 2a+2(r-a)=2r$.
If $r=s$, then $2r-1 \ge \gamma'_R(K_{r,s}) = w(f) \ge 2a +(r-a)^2$.
That is, $0 \ge (r-a-1)^2$.
This implies that $r-a-1=0$ and $\gamma'_R(K_{r,s}) = w(f) = 2r-1$.
\end{proof}

Notice that $K_{r,r}$ has maximum degree $\Delta=r$ and $n=2r$ vertices.
By Theorem \ref{bipartite}, $\gamma'_R(K_{r,r})=2r-1=\frac{2\Delta-1}{2\Delta}n$ which is the same as the upper bound $\lceil \frac{\Delta}{\Delta+1}n \rceil = \lceil 2r-2 +\frac{2}{r+1} \rceil = 2r-1$ in Conjecture \ref{conj1}.
While the gap between $\frac{2\Delta-1}{2\Delta} n$ and $\frac{\Delta}{\Delta+1} n$ being $\frac{\Delta-1}{2\Delta(\Delta+1)} n$, the reasons for the above values to be the same are $\Delta$ being close to $n$ and taking ceiling.
Similar situation happens for $K_{r,r+1}$, which has maximum degree $\Delta=r+1$ and $n=2r+1$ vertices.
By Theorem \ref{bipartite}, $\gamma'_R(K_{r,r+1})=2r=\frac{2\Delta-2}{2\Delta-1}n$ which is the same as $\lceil \frac{\Delta}{\Delta+1}n \rceil = \lceil 2r-1 +\frac{3}{r+2} \rceil = 2r$.
Also, the gap between $\frac{2\Delta-2}{2\Delta-1} n$ and $\frac{\Delta}{\Delta+1} n$ is $\frac{\Delta-2}{(\Delta+1)(2\Delta-1)} n$.

To get counterexamples, we modify complete bipartite graphs to obtain graphs whose $\Delta$ are far away from  $n$.
Consider the graph $G_{r,t}$ obtained from $t$ copies of $K_{r,r+1}$ by adding edges $y_{r+1}^i y_{1}^{i+1}$ for $1 \le i \le t$ with $y_{1}^{t+1}=y_{1}^1$, where the partite sets of the $i$-th $K_{r,r+1}$ are $X_i=\{x_1^i, x_2^i, \ldots, x_r^i\}$ and $Y_i=\{y_1^i, y_2^i, \ldots, y_{r+1}^i\}$.
See Figure \ref{fig1} for $G_{2,4}$.

\begin{figure}[htb]
\setlength{\unitlength}{0.1cm}
\begin{center}
\begin{picture}(110,10)(0,0)
 \put(00,00){\circle*{\ps}} \put(00,-3.7){$y_1^1$}
 \put(10,00){\circle*{\ps}}
 \put(20,00){\circle*{\ps}} \put(19,-3.7){$y_3^1$}
 \put(05,10){\circle*{\ps}} \put(05,10){\line(-1,-2){5}} \put(05,10){\line(01,-2){5}} \put(05,10){\line(03,-2){15}}
 \put(15,10){\circle*{\ps}} \put(15,10){\line(-3,-2){15}} \put(15,10){\line(-1,-2){5}} \put(15,10){\line(01,-2){5}}
 \put(20,00){\line(1,0){10}}
 \put(30,00){\circle*{\ps}} \put(29,-3.7){$y_1^2$}
 \put(40,00){\circle*{\ps}}
 \put(50,00){\circle*{\ps}} \put(49,-3.7){$y_3^2$}
 \put(35,10){\circle*{\ps}} \put(35,10){\line(-1,-2){5}} \put(35,10){\line(01,-2){5}} \put(35,10){\line(03,-2){15}}
 \put(45,10){\circle*{\ps}} \put(45,10){\line(-3,-2){15}} \put(45,10){\line(-1,-2){5}} \put(45,10){\line(01,-2){5}}
 \put(50,00){\line(1,0){10}}
 \put(60,00){\circle*{\ps}} \put(59,-3.7){$y_1^3$}
 \put(70,00){\circle*{\ps}}
 \put(80,00){\circle*{\ps}} \put(79,-3.7){$y_3^3$}
 \put(65,10){\circle*{\ps}} \put(65,10){\line(-1,-2){5}} \put(65,10){\line(01,-2){5}} \put(65,10){\line(03,-2){15}}
 \put(75,10){\circle*{\ps}} \put(75,10){\line(-3,-2){15}} \put(75,10){\line(-1,-2){5}} \put(75,10){\line(01,-2){5}}
 \put(80,00){\line(1,0){10}}
 \put(090,00){\circle*{\ps}} \put(089,-3.7){$y_1^4$}
 \put(100,00){\circle*{\ps}}
 \put(110,00){\circle*{\ps}} \put(110,-3.7){$y_3^4$}
 \put(095,10){\circle*{\ps}} \put(095,10){\line(-1,-2){5}} \put(095,10){\line(01,-2){5}} \put(095,10){\line(03,-2){15}}
 \put(105,10){\circle*{\ps}} \put(105,10){\line(-3,-2){15}} \put(105,10){\line(-1,-2){5}} \put(105,10){\line(01,-2){5}}
 \put(000,00){\line(0,-1){8}}
 \put(000,-8){\line(1,00){110}}
 \put(110,-8){\line(0,01){8}}
\end{picture}
\end{center}
\caption{The graph $G_{2,4}$.}
\label{fig1}
\end{figure}
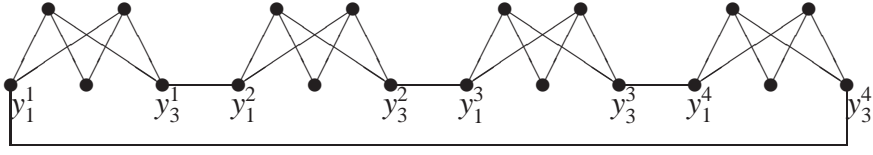

\begin{theorem} \label{Grt}
  If $r \ge 1$ and $t \ge 1$, then $\gamma'_R(G_{r,t})=2rt$.
\end{theorem}

\begin{proof}
The function $f$ defined by $f(x_j^i y_j^i)=2$ for $1 \le i \le t$ and $1 \le j \le r$, and $f(e)=0$ for all other edges $e$ is an edge Roman dominating function of weight $2rt$.
So $\gamma'_R(G_{r,t})\le 2rt$.

On the other hand, let $f$ be an edge Roman dominating function of $G_{r,t}$ with the minimum weight.
For every $1 \leq i \leq t$, let $a_i$ be the number of edges $e$ with $f(e)=2$ in the $i$-th $K_{r,r+1}$, and let $b_i=1$ if $f(y_{r+1}^i y_{1}^{i+1})=2$, and $b_i=0$ otherwise.
We define $b_0=b_t$.
Then $X_i$ has at least $\max\{0,r-a_i\}$ vertices and $Y_i$ has at least $\max\{0,r+1-a_i-b_i-b_{i-1}\}$ vertices that are not incident to any edge $e$ with $f(e)=2$.
Hence there are at least $\max\{0,r-a_i\} \max\{0,r+1-a_i-b_i-b_{i-1}\}$ edges $e'$ having $f(e')=1$.
These give
$$
    w(f) \ge \sum_{i=1}^t (2a_i+2b_i+\max\{0,r-a_i\} \max\{0,r+1-a_i-b_i-b_{i-1}\}) = \sum_{i=1}^t I_i,
$$
where $I_i = 2a_i+b_i+b_{i-1}+\max\{0,r-a_i\} \max\{0,r+1-a_i-b_i-b_{i-1}\}$.
It is sufficient to prove that $I_i \geq 2r$ for $1 \leq i \leq t$.
Suppose to the contrary that $I_i<2r$ for some $i$.
So $a_i<r$ and $r+1-a_i-b_i-b_{i-1} \geq 0$.
Then $I_i = 2a_i+b_i+b_{i-1}+(r-a_i)(r+1-a_i-b_i-b_{i-1})=2r+(r-a_i-1)(r-a_i-b_i-b_{i-1})$.
Observe that $(r-a_i-1)(r-a_i-b_i-b_{i-1}) \ge 0$, since either $r=a_i+1$ or $r-a_i-b_i-b_{i-1} \geq 0$.
So $I_i \geq 2r$ as desired.
\end{proof}

Notice that the graph $G_{r,t}$ has maximum degree $\Delta=r+1$ and $n=(2r+1)t$ vertices.
By Theorem \ref{Grt}, $\gamma'_R(G_{r,t})=2rt =\frac{2\Delta-2}{2\Delta-1} n > \frac{\Delta}{\Delta+1} n
= \lceil \frac{\Delta}{\Delta+1} n \rceil$ when $r \ge 2$ and $t$ a multiple of $r+2$.
This disproves Conjecture \ref{conj1}.
In fact, we shall prove that $\frac{2\Delta-2}{2\Delta-1}$ is asymptotic the optimal coefficient of $n$ for the upper bound of the edge Roman domination in Section 3.

\section{$k$-degenerate graphs}

Recall that a graph $G$ is $k$-degenerate if for every subgraph $H$ of $G$, the minimum degree  $\delta(H)$ of $H$ is at most $k$.
While the counterexamples in the previous section having the edge Roman domination numbers $\frac{2\Delta-2}{2\Delta-1} n$, this section shall prove that this is an upper bound for $k$-degenerate graphs.
It also establishes a close upper bound $\frac{2\Delta-2}{2\Delta-1} n + \frac{2}{2\Delta-1}$ for connected graphs.

We first need several useful lemmas that will be frequently applied in the rest of the paper.
A {\it removable triple} of a graph $G$ is a triple $(S,M_2, M_1)$, where $S$ is a nonempty subset of $V(G)$ and $M_2$ and $M_1$ are disjoint matchings in $G[S]$ such that every edge $e \in E(G)-M_1$ incident to a vertex in $S$ is adjacent to some edge in $M_2$.
We define the {\it ratio} $\rho(S,M_2,M_1)$ of a removable triple $(S,M_2,M_1)$  to be $\frac{2|M_2| + |M_1|}{|S|}$.


\begin{lemma} \label{remove}
  If a graph $G$ has a removable triple $(S,M_2,M_1)$ with $\rho(S,M_2,M_1) \le \alpha$, then $\gamma_R'(G) \le \gamma_R'(G-S) + \alpha |S|$.
\end{lemma}

\begin{proof}
Let $G' = G-S$ and let $f'$ be an edge Roman dominating function of $G'$ with the minimum weight.
Define a function $f\colon E(G) \to \{0,1,2\}$ by setting
$$
    f(e) = \left\{\begin{array}{ll}
         f'(e), & \mbox{ if } e\in E(G'); \\
         2,     & \mbox{ if } e\in M_2; \\
         1,     & \mbox{ if } e\in M_1; \\
         0,     & \mbox{ otherwise.}
                \end{array} \right.
$$
Suppose $e$ is an edge with $f(e)=0$.
If $e \in E(G')$, then $e$ is adjacent to an edge $e' \in E(G')$ with $f(e')=f'(e')=2$.
If $e \not\in E(G')$, then $e$ is incident to some vertex in $S$ and so by the definition of a removable triple $e$ is adjacent to some edge $e' \in M_2$ with $f(e')=2$.
Hence, $f$ is an edge Roman dominating function of $G$ and so $\gamma'_R(G) \le \gamma'_R(G') + 2 |M_2| + |M_1| \le \gamma'_R(G-S) + \alpha |S|$.
\end{proof}


\begin{lemma} \label{ratio}
  For every removable triple $(S,M_2,M_1)$ of $G$, if $\gamma_R'(G-S) \le \alpha |V(G-S)|$ but $\gamma_R'(G) > \alpha |V(G)|$, then $\rho(S,M_2,M_1) > \alpha$
\end{lemma}

\begin{proof}
Suppose to the contrary that $\rho(S,M_2,M_1) \le \alpha$ for some removable triple $(S,M_2,M_1)$ of $G$.
By Lemma \ref{remove}, $\gamma_R'(G) \le \gamma_R'(G-S) + \alpha |S| \le \alpha |V(G-S)| + \alpha |S| = \alpha |V(G)|$, a contradiction to the assumption that $\gamma_R'(G) > \alpha |V(G)|$.
\end{proof}


\begin{lemma}\label{mad}
  If $v$ is a vertex of degree $d$ in a graph $G$ and $M$ is a matching in $G[N(v)]$, then $G$ has a removable triple $(S,M_2,M_1)$ with $|S| \le 2d+1$ and
  $$
     \rho(S,M_2,M_1) \le \frac{2d-2|M|}{2d+1-2|M|} \le \frac{2d}{2d+1}.
  $$
\end{lemma}

\begin{proof}
Observe that $\frac{2d-2|M|}{2d+1-2|M|}$ decreases when $|M|$ increases.
By adding edges into $M$, we may without loss of generality assume that $M$ is a maximal matching in $G[N(v)]$.

Let $X=N(v)-V(M)$ and $Y=N(X)-N[v]$.
Since $M$ is a maximal matching in $G[N(v)]$, $X$ is stable in $G[N(v)]$.
We define $B=G[X \cup Y] - (E(G[X]) \cup E(G[Y]))$ and $M'$ to be a maximum matching of $B$.
Let $X' = X \cap V(M')$ and $X''=X-X'$; let $Y' = Y \cap V(M')$ and $Y''=Y-Y'$.
Notice that there are no edges between $X''$ and $Y''$, while possibly there are edges between $X'$ and $Y''$ and  edges between $X''$ and $Y'$.
See Figure \ref{fig2}.

Let $M_2 = M \cup M'$ and $S=N[v] \cup V(M_2)$.
If $\lvert X'' \rvert=0$, then $(S,M_2,\emptyset)$ is a removable triple with ratio $\frac{2|M_2|}{|S|} = \frac{2d-2|M|}{2d+1-2|M|}$.
If $\lvert X'' \rvert=1$, say $X'' = \{w\}$, then $(S,M_2,\{vw\})$ is a removable triple with ratio $\frac{2|M_2|+1}{|S|} = \frac{2d-1-2|M|}{2d-2|M|} < \frac{2d-2|M|}{2d+1-2|M|}$.
If $\lvert X'' \rvert \ge 2$, then for every $w \in X''$, $(S,M_2 \cup \{vw\},\emptyset)$ is a removable triple with ratio $\frac{2|M_2|+2}{|S|} = \frac{2d-2|M|-2|X''|+2}{2d-2|M|-|X''|+1} < \frac{2d-2|M|}{2d+1-2|M|}$.
\end{proof}


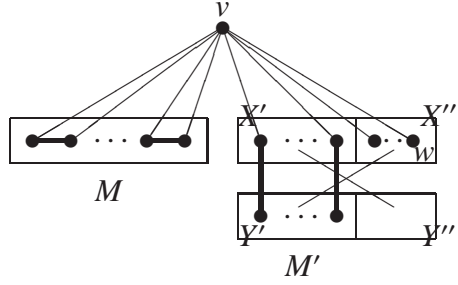
\begin{figure}[htb]
\setlength{\unitlength}{0.1cm}
\begin{center}
\begin{picture}(50,22)(0,-10)
 \put(00,00){\circle*{\ps}} \put(00,00){\line(1,0){5}} \put(00,0.2){\line(1,0){5}} \put(00,-0.2){\line(1,0){5}}
 \put(05,00){\circle*{\ps}} \put(08,-8){$M$}
 \put(08,-0.2){$\ldots$}
 \put(15,00){\circle*{\ps}} \put(15,00){\line(1,0){5}} \put(15,0.2){\line(1,0){5}} \put(15,-0.2){\line(1,0){5}}
 \put(20,00){\circle*{\ps}}
 \put(-3,-3){\line(1,0){26}} \put(-3,-3){\line(0,1){6}} \put(23,03){\line(-1,0){26}} \put(23,03){\line(0,-1){6}}
 \put(30,00){\circle*{\ps}} \put(30,00){\line(0,-1){10}} \put(30.2,00){\line(0,-1){10}} \put(29.8,00){\line(0,-1){10}} \put(30,-10){\circle*{\ps}}
 \put(33,-0.2){$\ldots$} \put(33,-10.2){$\ldots$}
 \put(40,00){\circle*{\ps}} \put(40,00){\line(0,-1){10}} \put(40.2,00){\line(0,-1){10}} \put(39.8,00){\line(0,-1){10}} \put(40,-10){\circle*{\ps}}
 \put(45,00){\circle*{\ps}}
 \put(46,-0.2){$\ldots$}
 \put(50,00){\circle*{\ps}} \put(50,-2.7){$w$}
 \put(27,-3){\line(01,0){26}} \put(27,-3){\line(0,1){6}}  \put(27,02){$X'$} \put(42.5,-3){\line(0,1){6}}
 \put(53,03){\line(-1,0){26}} \put(53,03){\line(0,-1){6}} \put(51,02){$X''$}
 \put(27,-13){\line(1,0){26}} \put(27,-13){\line(0,1){6}} \put(27,-14){$Y'$} \put(42.5,-13){\line(0,1){6}}
 \put(53,-7){\line(-1,0){26}} \put(53,-7){\line(0,-1){6}} \put(51,-14){$Y''$} \put(33,-18){$M'$}
 \put(25,15){\circle*{\ps}} \put(24,16.5){$v$}
                            \put(25,15){\line(-5,-3){25}} \put(25,15){\line(-4,-3){20}}
                            \put(25,15){\line(-2,-3){10}} \put(25,15){\line(-1,-3){05}}
                            \put(25,15){\line(01,-3){05}} \put(25,15){\line(01,-1){15}}
                            \put(25,15){\line(04,-3){20}} \put(25,15){\line(05,-3){25}}
 \put(35,-1.25){\line(05,-3){12.5}}
 \put(35,-8.75){\line(05,03){12.5}}
\end{picture}
\end{center}
\caption{The vertex $w$ exists only when $X'' \ne \emptyset$.}
\label{fig2}
\end{figure}


\begin{theorem} \label{k degenerate}
  If $G$ is a $k$-degenerate graph of $n$ vertices, then $\gamma'_{R}(G)\leq \frac{2k}{2k+1}n.$
\end{theorem}

\begin{proof}
The theorem clearly holds when $n=1$.
Suppose $G$ is a minimum counterexample to the theorem.
That is, $G$ is a $k$-degenerate graph $G$ with $\gamma_R'(G) > \frac{2k}{2k+1} |V(G)|$, but $\gamma_R'(H) \le \frac{2k}{2k+1} |V(H)|$ for every proper subgraph $H$ of $G$, which is also $k$-degenerate.
Since $G$ has a vertex of degree $d \le k$, Lemma \ref{mad} implies the existence of a removable triple of $G$ with ratio at most $\frac{2d}{2d+1} \le \frac{2k}{2k+1}$.
It is a contradiction to Lemma \ref{ratio}.
\end{proof}


We remark that every tree is $1$-degenerate, so the upper bound $\lfloor \frac{2n}{3} \rfloor$ for a tree of $n$ vertices \cite{Akbari1} is also a consequence of Theorem \ref{k degenerate}.
In addition, the result in \cite{Akbari1} on graphs of maximum degree $\Delta$ is a consequence of Theorem \ref{k degenerate}, since a graph of maximum degree $\Delta$ is $\Delta$-degenerate.
The objective of the rest of this section is to improve this bound in terms of the maximum degree for connected graphs.

\begin{lemma} \label{non-regular}
  Let $G$ be a graph of maximum degree $\Delta$ of $n$ vertices.
  If every component of $G$ contains a vertex of degree less than $\Delta$, then $\gamma'_R(G) \leq \frac{2\Delta-2}{2\Delta-1}n$.
\end{lemma}

\begin{proof}
Suppose that $G$ is a minimum counterexample to the lemma.
Since every component of $G-S$ contains a vertex of degree less than $\Delta$ for every $S \subseteq V(G)$, by Lemma \ref{ratio}, every removable triple of $G$ has ratio greater than $\frac{2\Delta-2}{2\Delta-1}$.
However, $G$ contains a vertex of degree less than $\Delta$.
So by Lemma \ref{mad}, there exists removable triple with ratio at most $\frac{2(\Delta-1)}{2(\Delta-1)+1}$, a contradiction.
\end{proof}

\begin{theorem}
  If $G$ is a connected graph of maximum degree $\Delta$ on $n$ vertices, then $\gamma'_R(G) \leq \frac{2\Delta-2}{2\Delta-1}n + \frac{2}{2\Delta-1}$.
\end{theorem}

\begin{proof}
According to Lemma \ref{mad}, $G$ has a removable triple $(S,M_2,M_1)$ with $|S| \le 2\Delta+1$ and $\rho(S,M_2,M_1) \le \frac{2\Delta}{2\Delta+1}$.
Since $G$ is connected, every component of $G-S$ contains a vertex of degree less than $\Delta$.
Therefore, by Lemma \ref{non-regular}, $\gamma'_R(G-S) \le \frac{2\Delta-2}{2\Delta-1} (n-|S|)$.
Then, by Lemma \ref{ratio}, $\gamma'_R(G) \leq \frac{2\Delta-2}{2\Delta-1} (n-|S|) + \frac{2\Delta}{2\Delta+1} |S| = \frac{2\Delta-2}{2\Delta-1}n + \frac{2}{4\Delta^2-1} |S| \leq \frac{2\Delta-2}{2\Delta-1}n + \frac{2}{2\Delta-1}$.
\end{proof}

The requirement for the connectivity of $G$ is necessary.
Consider the graph $t K_{\Delta,\Delta}$ of maximum degree $\Delta$ on $n=2\Delta t$ vertices.
By Theorem \ref{bipartite}, $\gamma_R'(t K_{\Delta,\Delta})=(2\Delta-1)t = \frac{2\Delta-1}{2\Delta}n
> \frac{2\Delta-2}{2\Delta-1}n + \frac{2}{2\Delta-1}$  when $t$ is large.

\section{Subcubic graphs}

Recall that Akbari \cite{Akbari1} showed that $\gamma'_R(G)\leq \frac{6}{7}n$ for every subcubic graph $G$ of $n$ vertices.
The main theorem of this section shows that $K_{3,3}$ is the only connected graph attaining this bound.
Note that Theorem \ref{subcubic} is tight as $\gamma'_R(G_{2,t}) = \frac{4}{5}n$ for every positive integer $t$, by Theorem \ref{Grt}.


\begin{theorem} \label{subcubic}
If $G$ is a subcubic graph of $n$ vertices contains no $K_{3,3}$ as a component, then $\gamma'_{R}(G)\leq \frac{4}{5}n$.
\end{theorem}

\begin{proof}
Suppose $G$ is a minimum counterexample to the theorem.
Then $G$ is connected.
By Lemma \ref{ratio}, every removable triple has ratio greater than $\frac{4}{5}$.
By Lemma \ref{mad}, $G$ has no vertex of degree at most two, so $G$ is cubic.

\noindent
{\bf Claim 1.}
$G$ is triangle-free.

\noindent
{\it Proof.}
Suppose to the contrary that $G$ has a triangle $v_0v_1v_2$.
Let $u_0$ be the third neighbor of $v_0$.
If $u_0$ is adjacent to both $v_1$ and $v_2$, then $G$ is $K_4$ and cannot be a counterexample.
So $u_0$ has a neighbor $w_0$ other than $v_1$ and $v_2$.
Then $(\{v_0,v_1,v_2,u_0,w_0\}, \{v_1v_2, u_0w_0\}, \emptyset)$ is a removable triple of ratio $\frac{4}{5}$, a contradiction.
$\Box$

Now, choose a shortest cycle $C\colon v_0 v_1 v_2 \ldots v_{|C|-1}v_0$ of length $|C| \not\equiv 1$ (mod 3).
Note that the existence of such a cycle follows from Theorem 1 in \cite{Chen}, which implies that every cubic graph has a cycle of length a multiple of 3.
In the following, the indices for the vertices in $C$ are taken modulo $|C|$.
By Claim 1, $|C| \ge 5$.

\noindent
{\bf Claim 2.}
(i)  If $|C|\equiv 2$ (mod 3), then $C$ has no chord.
(ii) If $|C|\equiv 0$ (mod 3), then $C$ has at most two chords.
Any chord of $C$ is of the form $v_a v_{a+3b}$; and if $C$ has two chords, then they are $v_av_{a+3b}$ and $v_{a+1}v_{a+1+3b}$ for some integers $a$ and $b$.

\noindent
{\it Proof.}
Suppose $C$ has a chord $e$, whose end vertices divide $C$ into two paths $Q_1$ and $Q_2$.
Each path $Q_r$ together with $e$ form a cycle $D_r$ of length shorter than $|C|$.
By the choice of $C$, each $|D_r| \equiv 1$ (mod 3) and so each $|Q_r| \equiv 0$ (mod 3).
This is possible only when $|C| \equiv 0$ (mod 3), which gives (i) and that a chord is of the form $v_a v_{a+3b}$.

Suppose $C$ has two chords $v_i v_{i+3j}$ and $v_{i'} v_{i'+3j'}$, say $0 = i < i' \le |C|-1$.
As $G$ is cubic, these two chords are disjoint.
If these two chords are non-crossing, say $0 = i < i+3j < i' < i'+3j' \le |C|-1$, then $v_i, v_{i+3j}, v_{i'}, v_{i'+3j'}$ divide $C$ into four paths $R_1, R_2, R_3, R_4$, where $C=v_iR_1v_{i+3j}R_2v_{i'}R_3v_{i'+3j'}R_4v_i$.
Since $R_1 \cup v_iv_{i+3j}$, $R_2 \cup R_4 \cup v_iv_{i+3j} \cup v_{i'}v_{i'+3j'}$ and $R_3 \cup v_{i'}v_{i'+3j'}$ are cycles shorter than $C$, $|R_1| \equiv |R_2|+|R_4|+1 \equiv |R_3| \equiv 0$ (mod 3).
But then $R_2 \cup R_3 \cup R_4 \cup v_iv_{i+3j}$ is a cycle shorter than $C$ with length $0$ (mod $3$), contradicting the choice of $C$.
If these two chords are crossing, say $0 = i < i' < i+3j < i'+3j' \le |C|-1$, then $v_i, v_{i'}, v_{i+3j}, v_{i'+3j'}$ divide $C$ into four paths $S_1, S_2, S_3, S_4$, where $C = v_iS_1v_{i'}S_2v_{i+3j}S_3v_{i'+3j'}S_4v_i$.
If the two chords are not of the desired form, then $S_1 \cup v_{i'}v_{i'+3j'} \cup S_3 \cup v_iv_{i+3j}$ is a cycle shorter than $C$, so $\lvert S_1 \rvert + \lvert S_3 \rvert \equiv 2$ (mod $3$).
But $\lvert S_1 \rvert \equiv \lvert S_3 \rvert$ (mod $3$), so $\lvert S_1 \rvert \equiv \lvert S_3 \rvert \equiv 1$ (mod $3$).
Similarly, $\lvert S_2 \rvert \equiv \lvert S_4 \rvert \equiv 1$ (mod $3$), so $\lvert C \rvert \equiv 1$ (mod $3$), a contradiction.

Finally, if there are three chords for which each pair is of the form $v_av_{a+3b}$ and $v_{a+1}v_{a+1+3b}$, then it is the case that $|C|=6$ and the chords are $v_0v_3, v_1v_4,v_2v_5$.
This implies that $G$ is in fact $K_{3,3,}$, violating the assumption of the theorem.
$\Box$

By Claim 2, we may assume that either $C$ has no chord, or $|C| \equiv 0$ (mod 3) and $C$ has one chord $v_1 v_{3a+1}$ or two chords $v_1 v_{3a+1}, v_2 v_{3a+2}$.
For any $v_i$ that is not an end of a chord of $C$, it has a neighbor $u_i$ not in $C$.
In particular, $u_0, u_3, \ldots, u_{3r}$ exist, where $3r=|C|-3$ when $|C| \equiv 0$ (mod 3) and  $3r=|C|-2$ when $|C| \equiv 2$ (mod 3).
In the following, when $u_i$ is mentioned we always assume that it exists.

\noindent
{\bf Claim 3.}
If $i \ne j$ but $u_i = u_j$, then $\min\{|i-j|, |C|-|i-j|\} =2$.

\noindent
{\it Proof.}
Vertices $v_i$ and $v_j$ divide $C$ into two paths $R_1$ and $R_2$.
If $\min\{|i-j|, |C|-|i-j|\} > 2$, then for each $r=1,2$, $R_r$ together with the path $v_iu_iv_j$ form a cycle $D_r$ of length shorter than $|C|$.
By the choice of $C$, each $|D_r| \equiv 1$ (mod 3) and so each $|R_r| \equiv
2$ (mod 3).
These imply that $|C| \equiv 1$ (mod 3), a contradiction.
$\Box$

By Claim 3,  $u_0, u_3, \ldots, u_{3r}$ are distinct except possibly $u_0=u_{3r}$ when $|C| \equiv 2$ (mod 3).
If $u_0=u_{3r}$, then $\lvert C \rvert \equiv 2$ (mod $3$), so $C$ is chordless, and $u_0,u_{-3},u_{-6},...,u_{-3r}$ exist and are distinct.
So we may without loss of generality assume that all $u_0, u_3, \ldots, u_{3r}$ are distinct.
Let
$$
    V_3 = \{v_0, v_3, \ldots, v_{3r}\} { ~~ \rm and ~~ }
    U_3 = \{u_0, u_3, \ldots, u_{3r}\}.
$$

\noindent
{\bf Claim 4.}
The vertex set $U_3$ is stable.

\noindent
{\it Proof.}
Suppose to the contrary that $u_{3a}$ is adjacent to $u_{3b}$ for some $0 \le 3a < 3b \le 3r$.
Vertices $v_{3a}$ and $v_{3b}$ divide $C$ into two paths $Q_1,Q_2$ with $|Q_1| = 3b-3a$ and $|Q_2|=|C|-3b+3a$.
For each $r=1,2$, path $Q_r$ together with the path $v_{3a}u_{3a}u_{3b}v_{3b}$ form a cycle $D_r$ with $|D_r| \equiv$ 0 or 2 (mod 3).
By the choice of $C$, $|C| \le 3b-3a+3$ and $|C| \le |C|-3b+3a+3$.  Hence $|C| \le 6$. Consequently, $a=0$ and $b=1$.
If $|C|=6$, then $(V(C) \cup U_3, \{u_0u_3, v_1v_2, v_4v_5\}, \emptyset)$ is a removable triple of ratio $\frac{6}{8} < \frac{4}{5}$, a contradiction.
If $|C|=5$, then $u_4$ exists and is distinct from $u_0,u_3$ by Claim 1, so $(V(C) \cup U_3 \cup \{u_{4}\}, \{u_0u_3, v_1v_2, v_4u_4\}, \emptyset)$ is a removable triple of ratio $\frac{6}{8} < \frac{4}{5}$, a contradiction.
$\Box$

Now, choose a maximal subset $U_3'$ of $U_3$ such that each $u_{3i} \in U_3'$ has a neighbor $w_{3i} \not\in V(C) \cup U_3$ and all such $w_{3i}$'s are distinct.
Let $U_3''=U_3-U_3'$ and $W_3' = \{w_{3i}\colon u_{3i} \in U_3'\}$.
If $|C| \equiv 0$ (mod 3), then let $S = V(C) \cup U_3 \cup W_3'$, $M_2 = \{v_{3i+1} v_{3i+2}\colon 0 \le i \le \frac{\lvert C \rvert}{3}-1\} \cup \{u_{3i} w_{3i}\colon u_{3i} \in U_3'\}$ and $M_1 = \{v_{3i} u_{3i}\colon u_{3i} \in U_3''\}$.
By the maximality of $U_3'$, $(S,M_2,M_1)$ is a removable triple.
However, $\rho(S,M_2,M_1) = \frac{\frac{4}{3}\lvert C \rvert - \lvert U_3'' \rvert}{\frac{5}{3}\lvert C \rvert - \lvert U_3'' \rvert} \leq \frac{4}{5}$, a contradiction.
Therefore, $|C| \equiv 2$ (mod 3).
By Claim 2, $u_{\lvert C \rvert-1}$ exists.
By Claim 3, $u_{\lvert C \rvert-1} \not\in U_3$.
Let $S' = V(C) \cup U_3 \cup \{u_{3r+1}\} \cup W_3'$, $M'_2 = \{v_{3i+1} v_{3i+2}\colon 0 \le i \le \frac{\lvert C \rvert-2}{3}-1\} \cup \{v_{3r+1} u_{3r+1}\} \cup \{u_{3i} w_{3i}\colon u_{3i} \in U_3'\}$ and $M'_1 = \{v_{3i} u_{3i}\colon u_{3i} \in U_3''\}$.
By the maximality of $U_3'$, $(S',M_2',M_1')$ is a removable triple.
However, $\rho(S',M_2',M_1') = \frac{\frac{4}{3}(\lvert C \rvert+1) - \lvert U_3'' \rvert}{\frac{5}{3}(\lvert C \rvert+1) - \lvert U_3'' \rvert} \leq \frac{4}{5}$, a contradiction.
\end{proof}

\section{Graphs on surfaces of small genus}

The first objective of this section is to prove Conjecture \ref{conj2}. 
A {\it surface} is a $2$-connected manifold.
Let $G$ be a graph and $\Sigma$ a surface.
Every connected component of $\Sigma-G$ is called a {\it face}.
We say that $G$ is {\it $2$-cell embeddable in $\Sigma$} if $G$ can be drawn in $\Sigma$ such that every face is homeomorphic to an open disk.

Let $G$ be a graph that is $2$-cell embeddable in a surface $\Sigma$.
We fix a $2$-cell embedding of $G$ in $\Sigma$.
We denote the set of faces of this embedding by $F(G)$.
Then for every face $f$ of this embedding, there exists a closed walk in $G$ that contains all edges incident with $f$.
We define the {\it degree} of a face $f$ to be the length of the shortest such walk.
We say that a vertex is a {\it $t$-vertex} if the degree of this vertex is $t$.
Similarly, we say that a face is a {\it $t$-face} if the degree of this face is $t$.


\begin{theorem} \label{planar}
  If $G$ is a graph of $n$ vertices that can be 2-cell embedded in the plane or the projective plane, then $\gamma_R'(G)\leq \frac{6}{7}n$.
\end{theorem}

\begin{proof}
The theorem is clearly true when $n=1$.
Suppose that $G$ is a counterexample with the minimum size of $|V(G)|$ to the theorem.
In particular, $G$ is connected.
Let $\Sigma$ be a surface in which $G$ can be $2$-cell embedded.
We fix a $2$-cell embedding of $G$ in $\Sigma$.
In addition, every removable triple of $G$ has ratio greater than $\frac{6}{7}$ by Lemma \ref{ratio}.
This implies that every vertex of $G$ has degree at least four by Lemma \ref{mad}.

If there exists a $4$-vertex $v$ incident to a 3-face, then $G[N(v)]$ has a matching $M$ of size one.
By Lemma \ref{mad}, there is a removable triple with ratio at most $\frac{2 \times 4 - 2 \times 1}{2 \times 4 + 1 - 2 \times 1}=\frac{6}{7}$, a contradiction.
Hence, no $4$-vertex is incident to a 3-face.

If there exists a $5$-vertex $v$ incident to at least three 3-faces, then $G[N(v)]$ has a matching $M$ of size two.
By Lemma \ref{mad}, there is a removable triple with ratio at most $\frac{2 \times 5 - 2 \times 2}{2 \times 5 + 1 - 2 \times 2} =\frac{6}{7}$, a contradiction.
So every $5$-vertex is incident to at most two $3$-faces.

Now we shall derive a contradiction by means of the discharging method.

For every $x \in V(G) \cup F(G)$, we define the charge $\ch(x)$ on $x$ to be $\deg(x)-4$.
According to Euler's formula, the sum of the charge is
$$\sum_{v\in V(G)} (\deg(v)-4) + \sum_{f \in F(G)}(\deg(f)-4) = -4|V|+4|E|-4|F|<0.$$

For every vertex $v$ incident to exactly $t$ $3$-faces with $t>0$, we move $\frac{\deg(v)-4}{t}$ units of charge to each $3$-face incident to it.
We denote the new charge on each $x \in V(G) \cup F(G)$ by $\ch'(x)$.
Clearly, $\sum_{x \in V(G) \cup F(G)} \ch(x) = \sum_{x \in V(G) \cup F(G)} \ch'(x)$.

We shall prove that $\ch'(x) \geq 0$ for every $x \in V(G) \cup F(G)$.
It is obviously true unless $x$ is a $3$-face.
Let $f$ be a $3$-face.
Note that $\ch(f)=-1$, and we proved that $f$ is not incident to any $4$-vertex.
Furthermore, as every $5$-vertex is incident to at most two $3$-faces, every $5$-vertex sends at least $\frac{1}{2}$ unit of charge to each $3$-face incident to it.
According to the discharing rule, $f$ receives at least $\frac{d-4}{d} \geq \frac{1}{3}$ units of charge from each $d$-vertex incident to it for $d\geq 6$, and receives at least $\frac{5-4}{2}=\frac{1}{2}$ units of charge from each 5-vertex incident to it.
Therefore, $\ch'(f) \geq 0$.
Consequently, $0>\sum_{x \in V(G) \cup F(G)} \ch(x) = \sum_{x \in V(G) \cup F(G)} \ch'(x) \geq 0$, a contradiction.
\end{proof}

The {\it girth} of a graph is the minimum length of a cycle in the graph. (The girth is infinite if the graph has no cycle.)
The other main theorem of this section is the following.
We improve the upper bound from $\frac{6}{7} n$ to $\frac{4}{5} n$ if we additionally assume the graph has girth at least five.
In fact, our result generalizes to surfaces of genus larger than the projective plane.


\begin{theorem} \label{planar_g5}
Let $\Sigma$ be the plane, projective plane, torus or Klein bottle.
If $G$ is a graph of girth at least 5 on $n$ vertices that can be 2-cell embedded in $\Sigma$, then $\gamma'_{R}(G)\leq \frac{4}{5}n$.
\end{theorem}

\begin{proof}
The theorem holds for $n=1$.
Suppose that $G$ is a counterexample with the minimum size of $|V(G)|$ to the theorem.
In particular, $G$ is connected.
Let $\Sigma$ be a surface of minimum genus in which $G$ can be 2-cell embedded.
We fix a $2$-cell embedding of $G$ in $\Sigma$.
By Lemma \ref{ratio}, every removable triple of $G$ has ratio greater than $\frac{4}{5}$.
So every vertex of $G$ has degree at least three by Lemma \ref{mad}.

Since the girth of $G$ is at least five, every $5$-face is surrounded by a cycle of length five.
We claim that every 5-face $f = (v_0,v_1,v_2,v_3,v_4)$ is incident to at most two 3-vertices.
Suppose to the contrary that $f$ is incident to at least three 3-vertices.
So two $3$-vertices incident with $f$, say $v_0$ and $v_2$, are non-adjacent.
Let $u_i \in N(v_i)-\{v_j\colon 0\leq j \leq 4\}$ for $0 \le i \le 2$.
Since $G$ has no 3-cycles and no 4-cycles, $u_0, u_1, u_2, v_0, v_1, v_2, v_3, v_4$ are eight distinct vertices.
Since $u_0$ and $u_2$ has degree at least three and $G$ has no $3$-cycles and no 4-cycles, we may choose $w_0$ and $w_2$ such that $w_i \in N(u_i)-\{v_j\colon 0 \leq j \leq 4\}$ for $i=0,2$ and $w_0 \ne w_2$.
See Figure \ref{fig3}.
Let $S = \{v_0,v_1,v_2,v_3,v_4,u_0,u_1,u_2,w_0,w_2\}$ and $M_2 = \{u_0w_0, v_1u_1, u_2w_2, v_3v_4\}$.
Then $(S, M_2, \emptyset)$ is a removable triple of ratio $\frac{8}{10} = \frac{4}{5}$, a contradiction.
This proves the claim.

\begin{figure}[htb]
\setlength{\unitlength}{0.1cm}
\begin{center}
\begin{picture}(40,17)(0,-7)
  \put(00,00){\circle*{\ps}} \put(00,-3){$u_0$} \put(00,00){\line(1,0){40}}
  \put(00,10){\circle*{\ps}} \put(0.2,7){$w_0$} \put(00,00){\line(0,1){10}} \put(-0.2,00){\line(0,1){10}}
  \put(0.2,00){\line(0,1){10}}
  \put(10,00){\circle*{\ps}} \put(10,-3){$v_0$} \put(10,00){\line(0,-1){10}}\put(10,-10){\circle*{\ps}}
  \put(10,-13){$v_4$}
  \put(10,-10){\line(1,0){20}} \put(10,-9.8){\line(1,0){20}} \put(10,-10.2){\line(1,0){20}}
  \put(20,00){\circle*{\ps}} \put(20,-3){$v_1$}
  \put(20,10){\circle*{\ps}} \put(20.2,7){$u_1$} \put(20,00){\line(0,1){10}} \put(19.8,00){\line(0,1){10}}
  \put(20.2,00){\line(0,1){10}}
  \put(30,00){\circle*{\ps}} \put(30,-3){$v_2$} \put(30,00){\line(0,-1){10}}\put(30,-10){\circle*{\ps}}
  \put(30,-13){$v_3$}
  \put(40,00){\circle*{\ps}} \put(40,-3){$u_2$}
  \put(40,10){\circle*{\ps}} \put(40.2,7){$w_2$} \put(40,00){\line(0,1){10}} \put(39.8,00){\line(0,1){10}}
  \put(40.2,00){\line(0,1){10}}
\end{picture}
\end{center}
\caption{A 5-face incident to at least three 3-vertices.}
\label{fig3}
\end{figure}
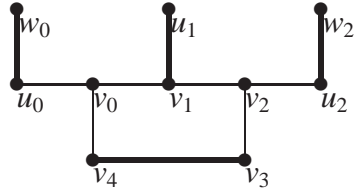

For every $x \in V(G) \cup F(G)$, define the charge $\ch(x)$ on $x$ to be $\deg(x)-4$.
According to Euler's formula, the sum of the charges is
$$
   \sum_{v\in V(G)} (\deg(v)-4) + \sum_{f \in F(G)}(\deg(f)-4) = -4|V|+4|E|-4|F|\leq 0.
$$

Now we describing the discharging rule.
We shall move charges of faces to vertices incident to it.
But we should notice that some face is not surrounded by a cycle.
For every face $f$, let $W_f$ be a shortest closed walk containing all edges incident with $f$, and let $W'_f$ be the walk obtained from $W_f$ by deleting the last vertex and the last edge.
For every vertex $v$ incident to $f$, define $t_{f,v}$ to be the number such that $v$ appears $t_{f,v}$ times in $W'_f$.
The discharging rule is that for every face $f$ incident to some $3$-vertex, move $\frac{\deg(f)-4}{\sum t_{f,u}}t_{f,v}$ units of charge to each $3$-vertex $v$ incident to $f$, where the summation in the denominator is over all $3$-vertices $u$ incident to $f$.
We denote the number of the units of new charges of $x$ by $\ch'(x)$.

We shall prove that $\ch'(x) \geq 0$ for every $x \in V(G) \cup F(G)$.
Observe that it is sufficient to prove this for $3$-vertices.
For every $3$-vertex $v$, $\ch(v)=-1$, and $v$ is incident to three faces of degree at least 5 by the assumption.
The vertex $v$ receives at least $\frac{d-4}{d}t_{f,v} \geq \frac{1}{3}t_{f,v}$ units of charge from a $d$-face incident to it for $d\geq 6$, and receives at least $\frac{5-4}{2}t_{f,v}=\frac{1}{2}t_{f,v}$ units of charge from each $5$-face incident to it by the previous claim.
Then $\ch'(v) \geq 0$ for every $3$-vertex $v$, since $\sum_{f \in F(G)} t_{f,v}=3$.

Since $0 \geq \sum_{x \in V(G) \cup F(G)} \ch(x) =\sum_{x \in V(G) \cup F(G)} \ch'(x) \geq 0$, we know that $\ch'(x)=0$ for every $x \in V(G) \cup F(G)$.
By Theorem \ref{subcubic}, there exists a vertex $v$ of degree at least four.
But $\ch'(v)>\deg(v)-4 \geq 0$ as $G$ has girth at least five.
So $0 \geq \sum_{x \in V(G) \cup F(G)} \ch'(x) >0$, a contradiction.
This proves that $\gamma'_R(G) \leq \frac{4}{5}n$.
\end{proof}

Based on Theorems \ref{planar} and \ref{planar_g5}, we expect the following conjecture holds.
Note that the upper bound in the following tends to $\frac{2}{3}$ when $k$ tends to infinity.
It is an evidence that supports the conjecture, since the behavior of a planar graph with large girth is like a tree.

\begin{conjecture}
  If $G$ is a planar graph of girth at least $3k+2$ on $n$ vertices, then $\gamma'_{R}(G)\leq \frac{2k+2}{3k+2}n$.
\end{conjecture}

\section{Graphs without $K_{2,3}$-subdivisions}

A graph is {\it outerplanar} if it can be embedded in the plane such that every vertex is incident to the infinite face.
Akbari et al.~\cite{Akbari2} showed that $\gamma'_R(G) \leq \frac{4}{5}n$ for every outerplanar graph of $n$ vertices.
In this section, we generalize the theorem to graphs without $K_{2,3}$-subdivisions, which is a proper superclass of outerplanar graphs.
Recall that $C_5$ attains the bound $\frac{4}{5}n$, and the coefficient $\frac{4}{5}$ of $n$ cannot be improved by excluding finitely many graphs.
We shall prove that the upper bound can be improved if no $5$-cycle in the graph can be separated from the rest of the graph by deleting at most one vertices.

A {\it subdivision} of a graph $H$ is a graph that can be obtained from $H$ by repeatedly deleting an edge $xy$ and adding a new vertex $z$ adjacent to $x,y$.
The following lemma is an immediate consequence of Kuratowski's theorem \cite{kk}.


\begin{lemma} \label{outer Kuratowski}
  A graph $G$ is an outerplanar graph if and only if $G$ does not contain a subgraph isomorphic to a subdivision of $K_4$ or $K_{2,3}$.
\end{lemma}

Let $G$ be an outerplanar graph.
We fix an embedding of $G$ in the plane such that every vertex is incident to the infinite face.
We define the {\it internal dual} graph $D(G)$ of $G$ to be the multigraph such that $V(D(G))$ is the set of faces of $G$ except the infinite face, and $E(D(G))=\{uv_e\colon u,v \in F(G), e \in E(G), e$ is incident to $u$ and $v\}$.
Note that $D(G)$ might not be simple by the definition.
But the following lemma shows that $D(G)$ is simple and is a tree, when $G$ is $2$-connected.


\begin{lemma} \label{internal dual}
  If $G$ is a $2$-connected outerplanar graph that is embedded in the plane such that every vertex is incident to the infinite face, then $D(G)$ is a tree.
\end{lemma}

\begin{proof}
Since $G$ is $2$-connected, $D(G)$ has no loops.
Suppose to the contrary that there is a cycle $v_1v_2\cdots v_nv_1$ in $D(G)$, where $n \geq 2$.
Then there is a vertex $u$ of $G$ inside the cycle, but $u$ is not incident to the infinite face, a contradiction.
\end{proof}

A {\it cut-vertex} in a graph is a vertex such that deleting this vertex from the graph results in at least two components.
A {\it block} $B$ in a graph $G$ is a maximal subgraph of $G$ of order at least two such that $B$ has no cut-vertex.
It is well-known that every graph has a block containing at most one cut-vertex.
And we call such a block an {\it end-block}.

\begin{theorem} \label{K_{2,3} cut}
  Let $G$ be a graph of $n$ vertices that does not contain a subgraph isomorphic to a subdivision of $K_{2,3}$.
  If $G$ does not contain $C_5$ as a component and there does not exist a vertex $v$ such that $G-v$ contains $C_5$ as a component, then $\gamma'_R(G) \leq \frac{3}{4}n$.
\end{theorem}

\begin{proof}
The theorem is true when $n \leq 4$.
We suppose that $G$ is a counterexample with the minimum size of $|V(G)|$.
So $G$ is connected and contains at least five vertices.

We say that a triple $(S,M_2,M_1)$ is {\it useful} if it is a removable triple with ratio at most $\frac{3}{4}$ such that $G[S]$ is connected and $G-S$ does not contains $C_5$ as a component.

\noindent
{\bf Claim 1.}
There does not exist a useful triple.

\noindent
{\it Proof.}
Suppose that $(S,M_2,M_1)$ is a useful triple such that $|S|$ is as large as possible.
As $G$ is a minimum counterexample, there exists a vertex $v$ such that $G-(S\cup\{v\})$ contains $C_5$ as a component.
Let $\C$ be the set of components of $G-(S\cup\{v\})$ isomorphic to $C_5$.
Since $G-v$ does not contain $C_5$ as a component, there exists an edge between $S$ and each member of $\C$ in $G$.
As $S$ is connected, we have that $|\C| \leq 2$, otherwise $G$ contains a subgraph isomorphic to a subdivision of $K_{2,3}$, a contradiction.
Let $S' = \{v\} \cup \bigcup_{C \in \C} V(C)$.
Clearly, there exists $M_2' \subseteq E(G[S'])$ such that $(S',M'_2,\emptyset)$ is a removable triple of $G-S$ with ratio at most $\frac{3}{4}$.
Observe that $G-(S \cup S')$ does not contain $C_5$ as a component.
So $(S \cup S', M_2 \cup M'_2, M_1)$ is a useful triple with $|S \cup S'| > |S|$, a contradiction.
$\Box$

\noindent
{\bf Claim 2.}
For every $S \subseteq V(G)$ such that $G[S]$ is connected and $G-S$ contains a component $C$ isomorphic to $C_5$, we have that $|N(S) \cap V(C)| \leq 2$ and $|N(V(C)) \cap S| \geq 2$.

\noindent
{\it Proof.}
First $|N(S) \cap V(C)| \leq 2$, otherwise $G$ contains a subgraph isomorphic to a subdivision of $K_{2,3}$, a contradiction.
Also $|N(V(C)) \cap S| \geq 2$, otherwise, either $G$ contains $C_5$ as a component, or $G$ contains a vertex such that deleting this vertex results in a component isomorphic to $C_5$.
$\Box$

\noindent
{\bf Claim 3.}
For every $S \subseteq V(G)$ such that $G[S]$ is connected and $\lvert N(V(G)-S) \rvert \leq 2$, there are at most two components of $G-S$ isomorphic to $C_5$.

\noindent
{\it Proof.}
By Claim 2, there exist two vertices $x,y \in S$ such that both $x,y$ have neighbors in each component of $G-S$ isomorphic to $C_5$.
Since $G$ does not contain a subgraph isomorphic to a subdivision of $K_{2,3}$, there are at most two components of $G-S$ isomorphic to $C_5$.
$\Box$

\noindent
{\bf Claim 4.}
No end-block of $G$ is isomorphic to $C_5$.

\noindent
{\it Proof.}
Suppose that $B$ is an end-block of $G$ isomorphic to $C_5$.
Let $v$ be the vertex in $B$ adjacent to a vertex not in $B$.
Let $u \in N(v)-V(B)$ and let $S=V(B) \cup \{u\}$.
Let $\C$ be the set of components of $G-S$ isomorphic to $C_5$.
By Claim 3, $\lvert \C \rvert \leq 2$.
For every member $C$ in $\C$, let $M_C$ be a maximal matching in $C$.
Then $(S \cup \bigcup_{C \in \C} V(C), \{uv\} \cup \bigcup_{C \in \C} M_C, \emptyset)$ is a removable triple with ratio $\frac{4(\lvert \C \rvert+1)}{5 \lvert C \rvert+6} \leq \frac{3}{4}$, since $\lvert \C \rvert \leq 2$.
So this removable triple is useful, a contradiction.
$\Box$

\noindent
{\bf Claim 5.}
No removable triple $(S,M_2,M_1)$ of $G$ with ratio at most $\frac{3}{4}$ such that $G[S]$ is connected and $\lvert N(V(G)-S) \rvert \leq 1$.

\noindent
{\it Proof.}
If there exists a removable triple $(S,M_2,M_1)$ of $G$ with ratio at most $\frac{3}{4}$ such that $G[S]$ is connected, and $\lvert N(V(G)-S) \rvert \leq 1$, then $(S,M_2,M_1)$ is useful by Claim 4, contradicting Claim 1.
$\Box$

\noindent
{\bf Claim 6.}
$G$ is not a cycle and no end-block of $G$ is a cycle.

\noindent
{\it Proof.}
Clearly, $G$ is not a cycle.
Suppose that some end-block $B$ of $G$ is a cycle.
By Claim 5, $\lvert V(B) \rvert \neq 5$.
Then it is easy to see that there exist two matchings $M_1,M_2$ of $B$ such that $M_2$ contains an edge incident with the vertex in $N(V(G)-V(B))$ and $(V(B),M_2,M_1)$ has ratio at most $\frac{3}{4}$, contradicting Claim 5.
$\Box$

\noindent
{\bf Claim 7.}
Every vertex of $G$ has degree at least two.

\noindent
{\it Proof.}
Let $v$ be a vertex of degree one, and let $u$ be the neighbor of $v$.
Since $G$ contains at least four vertices, $u$ has a neighbor $w$ other than $v$.
We assume that $w$ is chosen to minimize the number of components of $G-\{u,v,w\}$ isomorphic to $C_5$.
If $G-\{u,v,w\}$ does not contain $C_5$ as a component, then $(\{u,v,w\}, \{u,w\}, \emptyset)$ is an useful triple.
So $G-\{u,v,w\}$ contains $C_5$ as a component.
Note that for each component $C$ of $G-\{u,v,w\}$ isomorphic to $C_5$, $N(C) = \{u,w\}$, by Claim 2.
As we choose $w$ to minimize the number of components of $G-\{u,v,w\}$ isomorphic to $C_5$, there exists only one component $C$ of $G-\{u,v,w\}$ isomorphic to $C_5$ by Claim 3.
Hence, $(\{u,v,w\} \cup V(C), \{uw\} \cup M_2, \emptyset)$ is a useful triple, where $M_2$ is a maximal matching in $C$.
$\Box$

\noindent
{\bf Claim 8.}
No end block of $G$ is isomorphic to $K_4$.

\noindent
{\it Proof.}
Suppose that there exists an end-block $B$ of $G$ isomorphic to $K_4$.
Since $\lvert V(G) \rvert \geq 5$, $B \neq G$.
Let $v$ be the vertex in $N(V(G)-V(B))$.
And let $e$ be an edge of $B$ incident with $v$ and $e'$ the edge of $B$ not adjacent to $e$.
Then $(V(B),\{e\},\{e'\})$ is a removable triple with ratio $\frac{3}{4}$, contradicting Claim 5.
$\Box$

\noindent
{\bf Claim 9.}
Every end-block of $G$ is outerplanar but not an edge.

\noindent
{\it Proof.}
Let $B$ be an end-block of $G$.
By Claim 7, $B$ is not an edge.
Suppose that $B$ is not outerplanar.
By Lemma \ref{outer Kuratowski}, $B$ contains a subgraph isomorphic to a subdivision of $K_4$.
Let $H$ be a subgraph of $B$ isomorphic to a subdivision of $K_4$.
If $\lvert V(H) \rvert \geq 5$, then $H$ contains a subgraph isomorphic to a subdivision of $K_{2,3}$, a contradiction.
So $\lvert V(H) \rvert =4$ and $H=K_4$.
But by Claim 8, $B \neq K_4$, so there exists $v \in V(B)-V(H)$.
Since $B$ is $2$-connected, there exist two paths in $B$ from $v$ to $V(H)$ only intersecting in $v$.
However, it implies that $B$ contains a subgraph isomorphic to a subdivision of $K_{2,3}$, a contradiction.
$\Box$

\noindent
{\bf Claim 10.}
There does not exist a path of four vertices in $G$ such that every vertex is of degree two in $G$.

\noindent
{\it Proof.}
Let $P=v_1v_2v_3v_4$ be a path of four vertices in $G$ such that every vertex in $P$ has degree two in $G$.
Let $v \in N(v_1)-V(P)$.
Note that $v_4$ is not adjacent to $v$, otherwise, $G$ contains an end-block of $G$ isomorphic to the $5$-cycle, contradicting Claim 4.
Let $G'$ be the graph obtained from $G-\{v_1,v_2,v_3\}$ by adding the edge $vv_4$.
It is easy to see that $\gamma'_R(G) \leq \frac{3}{4} |V(G)|$ if $\gamma'_R(G') \leq \frac{3}{4} |V(G')|$.
As $G$ is a minimum counterexample and $G'$ is connected, either $G'=C_5$ or there exists a vertex $w$ in $G'$ such that $G-w$ contains a component isomorphic to $C_5$.
For the former, $G$ is the $8$-cycle; for the latter, $C_8$ is an end-block of $G$.
Both cases contradict Claim 6.
$\Box$

Let $B$ be an end-block of $G$.
By Claim 9, $B$ is outerplanar but not an edge.
We fix an embedding of $B$ such that all vertices are incident with the infinite face.
Let $T$ be the internal dual of $B$.
By Claim 6, $T$ contains at least two vertices.
For every $t \in V(T)$, let $f_t$ be the face of $B$ corresponding to $t$.
If $G \neq B$, let the root of $T$ be a vertex $t$ such that $f_t$ contains the vertex in $N(V(G)-V(B))$; otherwise, let the root of $T$ be an arbitrary vertex.

\noindent
{\bf Claim 11.}
For every non-root leaf $t$ of $T$, the boundary of $f_t$ is a $5$-cycle.

\noindent
{\it Proof.}
Let $S$ be the boundary cycle of $f_t$.
By Claim 10, $\lvert V(S) \rvert \leq 5$.
Suppose that $S$ is a $3$-cycle or a $4$-cycle.
Let $\C$ be the set of components of $G-V(S)$ isomorphic to $C_5$.
$\lvert \C \rvert \geq 1$, otherwise there exists an useful triple $(V(S),M_2,M_1)$ for some $M_2$ and $M_1$.
By Claim 3, $|\C| \leq 1$.
Let $C$ be the member of $\C$.
Note that there does not exist a component $B$ of $G-(V(S) \cup V(C))$ such that $|N(C') \cap (V(S) \cup V(C))| \geq 2$, otherwise $B$ contains a subgraph isormorphic to a subdivision of $K_{2,3}$.
As $B$ is $2$-connected, $B=G[V(S) \cup V(C)]$.
But $(V(B),M_2',M_1')$ is a removable triple with ratio at most $\frac{3}{4}$ such that $B$ is connected and $\lvert N(V(G)-V(B)) \rvert \leq 1$.
It is a contradiction to Claim 5.
$\Box$

Let $s$ be a leaf of $T$ that is as far as from the root of $T$ as possible, and let $p$ be the neighbor of $s$ in $T$.
Let $S$ be the subset of $V(G)$ such that $G[S]$ is the union of the boundary of $f_p$ and the boundary of $f_c$ for each child $c$ of $p$.
Note that each $c$ is a leaf by our choice of $p$.
Also, $\lvert V(f_p \cap f_c) \rvert =2$ for every child $c$ of $p$.
Let $Q$ be the boundary cycle of $f_p$.

Clearly, $G[S]$ is connected.
Suppose that there exists a component $C$ of $G-S$ isomorphic to $C_5$.
Then $N(C) \subseteq S$.
Since $B$ is $2$-connected, $\lvert N(C) \rvert \geq 2$.
By Claim 2, $\lvert N(C) \rvert=2$.
Since $c$ is a leaf for every child $c$ of $p$, $N(C) \subseteq V(f_p)$.
Note that $C$ bounds a face that corresponds to a leaf of $T$.
Also, $C$ does not contain the vertex in $N(V(G)-V(B))$.
So the leaf of $T$ corresponds to $C$ is farther than $s$ from the root of $T$, a contradiction.
So $G-S$ does not contain $C_5$ as a component.
Note that $N(V(G)-S) \subseteq V(Q)$.
Since $s$ is the leaf of $T$ farthest from the root, $\lvert N(V(G)-S) \rvert \leq 2$.
Since $G$ does not contain a subgraph isomorphic to a subdivision of $K_{2,3}$, if $\lvert N(V(G)-S) \rvert \leq 2$, then there two vertices are adjacent.
Let $M$ be a maximal matching of the minimum size such that $M$ contains an edge incident with all vertices in $N(V(G)-S)$.
If $|V(Q)| \not \equiv 1$ (mod $3$), then let $M_2=M$ and $M_1=\emptyset$; otherwise, pick an edge $e$ in $M$ not incident with a vertex in $N(V(G)-S))$, and let $M_1=\{e\}$ and $M_2 = M-M_1$.
For every child $c$ of $p$, there exists an edge $e_c$ not incident with any vertex of $Q$ such that $G[S]-(M \cup \bigcup_{c}e_c)$ has no edges, where the union runs through all children $c$ of $p$.
Let $M^*=\bigcup_{c}e_c$, where the union runs through all children $c$ of $p$.
If $\lvert V(Q) \rvert \equiv 0$ (mod $3$), then the ratio of $(S,M_2 \cup M^*,M_1) = \frac{2}{3}$ .
If $\lvert V(Q) \rvert \equiv 2$ (mod $3$), then $\lvert S \rvert \geq \lvert Q \rvert +3 \geq 8$, so the ratio of $(M,M_2 \cup M^*,M_1)$ is $\frac{\frac{2}{3} \lvert V(Q) \rvert + \frac{2}{3} + 2\lvert M^* \rvert}{\lvert V(Q) \rvert + 3 \lvert M^* \rvert} \leq \frac{3}{4}$.
By Claim 1, $\lvert V(Q) \rvert \equiv 1$ (mod $3$).

Note that $\frac{2\lvert M_2 \rvert + \lvert M_1 \rvert}{\lvert S \rvert} = \frac{\frac{2}{3}\lvert V(Q) \rvert + 2 \lvert M^* \rvert+1}{\lvert V(Q) \rvert + 3 \lvert M^* \rvert} = \frac{2}{3} + \frac{1}{3(\lvert V(Q) \rvert + 3 \lvert M^* \rvert)} \leq \frac{3}{4}$ since $\lvert V(Q) \rvert \geq 4$.
So if we can choose the edge in $M_1$ such that this edge is not incident with the boundary of $f_c$ for some child $c$ of $p$, then $(S,M_2 \cup M^*,M_1)$ is a removable triple and hence is useful.
Therefore, $|M^*| \geq \frac{|V(Q)|-1}{3}$.
On the other hand, $(S,M_2 \cup M^*,\emptyset)$ is a removable triple with ratio $\frac{2}{3} + \frac{4}{3(\lvert V(Q) \rvert + 3\lvert M^* \rvert)}$, so $16>\lvert V(Q) \rvert + 3 \lvert M^* \rvert \geq 2 \lvert V(Q) \rvert-1$.
Hence, $\lvert V(Q) \rvert = 4$ or $7$.
Similarly, if $S=V(G)$, then $\lvert M^* \rvert \geq \frac{\lvert V(Q) \rvert+2}{3}$, so $\lvert V(Q) \rvert=4$ and $2 \leq \lvert M^* \rvert \leq 3$, but it is easy to check that $\gamma'_R(G) \leq \frac{3}{4}n$ in this case.
Consequently, $S \neq V(G)$, and either $\lvert V(Q) \rvert=4$ and $1 \leq \lvert M^* \rvert \leq 3$, or $\lvert V(Q) \rvert=7$ and $\lvert M^* \rvert=2$.

Denote $Q_1$ by $v_1v_2\ldots v_{|V(Q)|}v_1$.
Without loss of generality, we may assume that $\{v_1\} \subseteq N(V(G)-S) \subseteq \{v_1,v_2\}$.
Let $w$ be a vertex in $N(V(G)-S) \cap N(v_1)$ such that $G-(S \cup \{w\})$ has as less components isomorphic to $C_5$ as possible.
Let $S' = S \cup \{w\}$ and let $N=\{v_{3i+2}v_{3i+3}\colon 0 \leq i \leq \frac{|V(Q)|-1}{3}-1\}$.
For every child $c$ of $p$, there exists an edge $e'_c$ such that $G[S']-(N \cup \{v_1w\} \cup \bigcup_{c}e'_c)$ has no edges, where the union runs through all children $c$ of $p$.
Note that $(S',N \cup \{v_1w\} \cup \bigcup_{c}e'_c, \emptyset)$ is a removable triple with ratio at most $\frac{3}{4}$, where the union runs through all children $c$ of $p$, since $|M^*| \geq \frac{|V(Q)|-1}{3}$.
Therefore, $G-(S \cup \{w\})$ contains a component isomorphic to $C_5$.
Note that $w \in N(C)$ for every component $C$ of $G-(S \cup \{w\})$ isomorphic to $C_5$.
If there exists a component $C$ of $G-(S \cup \{w\})$ isomorphic to $C_5$ satisfies that $N(C) \subseteq \{w,v_1\}$, then there exists $w' \in N(V(G)-S) \cap N(v_1)$ such that $G-(S \cup \{w'\})$ has no component isomorphic to $C_5$, contradicting the choice of $w$.
So every component $C$ of $G-(S \cup \{w\})$ isomorphic to $C_5$ satisfies that $N(C)=\{w,v_2\}$ by Claim 2.
But in this case, there is at most one such component, otherwise $G$ contains a subgraph isomorphic to a subdivision of $K_{2,3}$.
Then there exists $w'' \in V(C) \cap N(v_2)-S$ such that $G-(S \cup \{w''\})$ has no components isomorphic to $C_5$.
Define $N'=\{v_1v_{\lvert V(Q) \rvert}, v_{3i+1}v_{3i+2}: 1 \leq i \leq \frac{\lvert V(Q) \rvert-1}{3}-1 \}$.
For every child $c$ of $p$, there exists an edge $e''_c$ such that $G[S']-(N' \cup \{v_2w''\} \cup \bigcup_{c}e''_c)$ has no edges, where the union runs through all children $c$ of $p$.
Then $(S \cup \{w''\}, N' \cup \{v_2w''\} \cup \bigcup_{c} e''_c, \emptyset)$ is a useful triple, where the union runs through all children $c$ of $p$, a contradiction.
This proves the theorem.
\end{proof}

\begin{theorem}
$\gamma'_R(G) \leq \frac{4}{5}n$ for every graph $G$ on $n$ vertices containing no subgraph isomorphic to a subdivision of $K_{2,3}$.
\end{theorem}

\begin{proof}
Suppose that $G$ is a counterexample with the minimum size of $|V(G)|$ of this theorem.
By Theorem \ref{K_{2,3} cut}, either $G$ contains $C_5$ as a component, or there exists $v$ such that $G-v$ contains $C_5$ as a component.
For the former, let $C$ be a component of $G$ isomorphic to $C_5$, then $(V(C),M_2,\emptyset)$ is a removable triple of $G$ with ratio $\frac{4}{5}$, where $M_2$ is a maximal matching of $C_5$.
For the latter, let $C'$ be the component of $G-v$ isomorphic to $C_5$, then let $e$ be an edge of $G$ with end $v$ and a vertex $v'$ of $C'$, and let $M_2'$ be the maximal matching of $C'-v'$ of size one, then $(V(C') \cup \{v\}, M_2' \cup \{e'\}, \emptyset)$ is a removable triple of $G$ with ratio less than $\frac{4}{5}$.
Either case contradict Lemma \ref{ratio}.
This proves the theorem.
\end{proof}


\end{document}